\documentclass[reqno]{amsart}
\usepackage{times}
\usepackage{amsmath,amssymb,amsfonts}
\usepackage{graphicx}
\usepackage{color}
\usepackage[latin1]{inputenc}
\usepackage[all]{xy}
\usepackage[dvips]{hyperref}
\usepackage{epsfig}
\usepackage{pstricks,pst-grad}
\usepackage{array}
\newcommand{\RM}[1]{\MakeUppercase{\romannumeral #1{.}}}

\usepackage{verbatim}
\usepackage{enumerate}

\newcommand {\QQ}{{\mathbb Q}}
\newcommand {\RR}{{\mathbb R}}
\newcommand {\ZZ}{{\mathbb Z}}

\DeclareMathOperator {\mult}{mult}
\DeclareMathOperator {\dist}{dist}

\DeclareMathOperator {\lab}{lab}
\DeclareMathOperator {\trop}{trop}
\DeclareMathOperator {\ev}{ev}

\DeclareSymbolFont {mysymbols}{OMS}{cmsy}{m}{n}
\DeclareMathSymbol {\calI}{\mathalpha}{mysymbols}{`I}
\DeclareMathSymbol {\calK}{\mathalpha}{mysymbols}{`K}
\DeclareMathSymbol {\calO}{\mathalpha}{mysymbols}{`O}
\DeclareMathSymbol {\calT}{\mathalpha}{mysymbols}{`T}

\DeclareSymbolFont {mylargesymbols}{OMX}{cmex}{m}{n}
\DeclareMathSymbol {\dunion}{\mathop}{mylargesymbols}{"60}

\newcommand {\calM}{{\mathcal M}}
\newcommand {\df}[1]{\textsl {#1}}

\newcommand {\preprint}[2]{preprint \discretionary {#1/}{#2}{#1/#2}}

\newcommand {\twolines}[2]{\stackrel{\scriptstyle {#1}}{#2}}

\renewenvironment {enumerate}%
  {\rule{1mm}{0mm}\begin {oldenumerate}%
    \parskip1ex plus0.5ex \itemsep 0mm \parindent 0mm}%
  {\end {oldenumerate}}

\renewenvironment {itemize}%
  {\rule{1mm}{0mm}\begin {olditemize}%
    \parskip1ex plus0.5ex \itemsep 0mm \parindent 0mm}%
  {\end {olditemize}}

\theoremstyle {plain}
\newtheorem {theorem}{Theorem}[section]
\newtheorem {proposition}[theorem]{Proposition}

\newtheorem {lemma}[theorem]{Lemma}
\newtheorem {corollary}[theorem]{Corollary}
\theoremstyle {definition}
\newtheorem {definition}[theorem]{Definition}
\theoremstyle {remark}
\newtheorem {remark}[theorem]{Remark}
\newtheorem {remdef}[theorem]{Remark and Definition}
\newtheorem {example}[theorem]{Example}
\newtheorem {construction}[theorem]{Construction}

\newtheorem{explanation}[theorem]{Explanation}

\begin{document}
\title [Tropical orbit spaces and the moduli spaces of elliptic tropical curves]{Tropical
  orbit spaces and the moduli spaces of elliptic tropical curves}
\author {Matthias Herold}

\address {Matthias Herold, Fachbereich Mathematik, TU Kaiserslautern, Postfach
  3049, 67653 Kaiserslautern, Germany}
\email {herold@mathematik.uni-kl.de}

\begin {abstract}
  We give a definition of tropical orbit spaces and their morphisms. We show that, under certain conditions, the weighted number of preimages of a point in the target of such a morphism does not depend on the choice of this point. We equip the moduli spaces of elliptic tropical curves with a structure of tropical orbit space and, using our results on tropical orbit spaces, simplify the known proof of the fact that the weighted number of plane elliptic tropical curves of degree $d$ with fixed $j$-invariant which pass through $3d-1$ points in general position in $\RR^2$ is independent of the choice of a configuration of points.
\end {abstract}

\maketitle

\section{introduction}
Tropical geometry is a relatively new mathematical domain. It has applications in several branches of mathematics and, in particular, has been used for solving various enumerative problems. One of the first results concerning enumerative problems in this domain was achieved by G.~Mikhalkin in \cite{M1}. He established an important correspondence between complex algebraic curves satisfying certain constraints and tropical analogs of these curves. This correspondence theorem was reproven in slightly different forms in \cite{NS}, \cite{Sh} and \cite{ST}.
Mikhalkin's results initiated the study of enumerative problems in tropical geometry (see for example \cite{GM1}, \cite{GM2}, \cite{GM3}). Dealing with counting problems, it is naturally to work with moduli spaces. The first step in this direction was the construction of the moduli spaces of rational curves given in \cite{M2} and \cite{GKM}. In \cite{GKM} the authors developed some tools to deal with enumerative problems for rational curves, using the notation of tropical fan. They introduced morphisms between tropical fans and showed that, under certain conditions, the weighted number of preimages of a point in the target of such a morphism does not depend on the chosen point. After showing that the moduli spaces of rational tropical curves have the structure of a tropical fan, they used this result to count rational curves passing through given points.

Following their approach, we introduce similar tools for enumerative problems concerning tropical curves of genus $1$. Considering moduli spaces of elliptic tropical curves, it is natural to expect the appearance of a counterpart of stacks in the tropical setting. Since we are mainly interested in the quotient structure of the moduli spaces it is clear that the definition for the counterpart of stacks, given in this paper, will not be the final one. Therefore, we call our objects tropical orbit spaces instead of tropical stacks. The definition is given in the second chapter. With the help of this definition, we develop some tools for dealing with tropical enumerative problems in genus $1$. The main statement of the second chapter is Corollary \ref{cor-image} which states that, for surjective morphisms between tropical orbit spaces of the same dimension such that the target is irreducible, the number of preimages (counted with multiplicities) is the same for each general choice of a point. The corollary can be used to prove invariance in tropical enumerative problems in genus $1$. In chapter $3$ we show that the moduli spaces of elliptic tropical curves with fixed $j$-invariant have a structure of a tropical orbit space. Then, we use the tools elaborated in chapter $2$ for the enumerative problems of counting elliptic tropical curves with fixed $j$-invariant (these problems were first considered in \cite{KM}), and simplify the proof of one of the main results in \cite{KM}.


The author would like to thank Andreas Gathmann and Johannes Rau for the introduction to the problem and he would like to thank Andreas Gathmann and Ilia Itenberg for many helpful discussions.

\section{ tropical orbit space}

In this part we denote a finitely generated free abelian
group by $\Lambda$ and the corresponding real vector space $
\Lambda \otimes_{\ZZ} \RR$ by $V$. So we can consider $\Lambda$ as a
lattice in $V$. The dual lattice in the vector space $V^{\vee}$ is denoted by $\Lambda^{\vee}$.

\begin{definition}[General and closed cones] \label{def-cones}
  A $general\ cone$ in $V$ is a subset $\sigma \subseteq V$ that can be
  described by finitely many linear integral equalities,
  inequalities and strict inequalities, i.e. a set of the form
  \[ \sigma = \{ x \in V | f_1(x)=0, \dots, f_r(x)=0, f_{r+1}(x) \geq 0, \dots, f_{r+s}(x) \geq 0, \] \[  f_{r+s+1}(x)>0, \dots, f_{N}(x)>0 
  \} \tag{$*$}\]

  for some linear forms $f_1, \dots, f_N \in \Lambda^{\vee}$. We
  denote by $V_\sigma$ the smallest linear subspace of $V$
  containing $\sigma$ and by $\Lambda_\sigma$ the lattice
  $V_\sigma \cap \Lambda$. We define the $dimension$ of
  $\sigma$ to be the dimension of $V_\sigma$.
  We call $\sigma$ a $closed\ cone$ if there are no strict inequalities in $ (*) $ (i.e. if $N=r+s$).
\end{definition}

\begin{definition}[Face]
A $face$ of $\sigma$ is a general cone $\tau \subset \sigma$ which can be obtained from $\sigma$ by changing some of the non-strict inequalities in $ (*) $ to equalities. 
\end{definition}

\begin{definition} [Fan and general fan]
A $fan$ in $V$ is a set $X$ of closed cones in $V$ such that 
\vspace{-0.5 cm}
\begin{enumerate}
\item each face of a cone in $X$ is also a cone in $X$;
\item the intersection of any two cones in $X$ is a face of each of them.
\end{enumerate}
A {\it general fan} in $V$ is a set $\tilde{X}$ of general cones in $V$ satisfying the following property: there exist a fan $X$ and a subset $R \subset X$ such that $\tilde{X}=\{\tau \setminus U|\ \tau \in X\}$, where $U= \bigcup_{\sigma \in R} \sigma$. We put $|\tilde{X}|=\bigcup_{\tilde{\sigma}\in \tilde{X}}\tilde{\sigma}$.
A (general) fan is called {\it pure-dimensional}, if all its inclusion-maximal cones are of the same dimension. In this case we call the highest dimensional cones \df{facets}. The set of $n-$dimensional cones of a (general) fan $X$ is denoted by $X^{(n)}$.
\end{definition}

\begin{construction}[Normal vector] \label{normal}
If $\emptyset \not=\tau,\ \sigma$ are cones in $V$ and $\tau$ is a subcone of $\sigma$ such that $\dim \tau = \dim \sigma -1$, then there is a non-zero linear form $g \in \Lambda^{\vee}$, which is zero on $\tau$ and  positive on $\sigma \backslash \tau$. Then $g$ induces an isomorphism $V_{\sigma}/ V_{\tau}\cong \RR$.
There exists a unique generator $u_{\sigma / \tau} \in \Lambda_{\sigma}/ \Lambda_{\tau}$, lying in the same half-line as $\sigma/ V_{\tau}$ and we call it the primitive $normal\ vector$ of $\sigma$ relative to $\tau$.
In the following we write $\tau\leq\sigma$ if $\tau$ is a subcone of $\sigma$ and $\tau< \sigma$ if $\tau$ is a proper subcone of $\sigma$.
\end{construction}

\begin{definition}[Weighted and general tropical fans]
A $weighted\ fan$ $\left( X,\omega_X \right)$ in $V$ is a pure-dimensional general fan $X$ of dimension $n$ with a map $\omega_X : X^{\left( n \right)} \rightarrow \QQ$. The numbers $\omega_X \left( \sigma \right)$ are called $weights$ of the general cones $\sigma \in X^{\left( n \right)}$. By abuse of notation we also write $\omega$ for the map and $X$ for the weighted fan.\\
A $general\ tropical\ fan$ in $V$ is a weighted fan $\left( X,\omega_X \right)$ fulfilling the balancing condition
$$ \sum_{\sigma > \tau} \omega_X \left( \sigma \right) \cdot u_{\sigma / \tau}=0 \quad \in V/V_{\tau}  $$
for any $\tau \in X^{\left( \dim X-1 \right)}$.
\end{definition}

\begin{definition}[Open fans] \label{defn-openfan}
  Let $\widetilde{F}$ be a general fan in
  $\RR^n$ and $0\in U \subseteq \RR^n$ an open subset. The set
  $F= \widetilde{F} \cap U = \{\sigma \cap U | \sigma \in \widetilde{F} \}$ is called an \df{open fan} in
  $\RR^n$. As in the case of fans, put $|F|= \bigcup_{\sigma' \in F}
  \sigma'$.\\
If $\widetilde{F}$ is a general weighted fan, we call $F$ a {\it weighted open fan}.
\end{definition}

\begin{remark}
 Since $0\in U$ is open, $\tilde{F}$ is defined by $F$.
\end{remark}

\begin{definition}[General polyhedron]
A \df{general polyhedron} is a set $\sigma \subset \RR^n$ such that there exists a rational polyhedron $\tilde{\sigma}$ and a union $u$ of faces of $\sigma$ such that $\sigma=\tilde{\sigma}\backslash u$. (This definition is equivalent to saying that the faces have the following form $\{ x \in V | f_1(x)=p_1, \dots,$ $f_r(x)=p_r,$ $f_{r+1}(x) \geq p_{r+1}, \dots,$ $f_{r+s}(x) \geq p_{r+s},$ $f_{r+s+1}(x)>p_{r+s+1}, \dots,$ $f_{N}(x)>p_{N} 
  \}$ for some linear forms $f_1, \dots, f_N \in \ZZ^n$ and numbers $p_1,\cdots, p_N \in \RR$.)
\end{definition}

\begin{definition}[General polyhedral precomplexes]\label{defn-prepolycomplex}
A \df{(general) polyhedral precomplex} is a topological space $|X|$ and a set $X$ of subsets of $|X|$ equipped with embeddings $\varphi_{\sigma}: \sigma \rightarrow \RR^{n_{\sigma}}$ for all $\sigma \in X$ such that
\vspace{-0.1 cm}
\begin{enumerate}
    \item $X$ is closed under taking intersections, i.e. $\sigma \cap \sigma' \in X$ is a face of $\sigma$ and of $\sigma'$ for any $\sigma, \sigma' \in X$ such that $\sigma \cap \sigma' \neq
    \emptyset$,
    \item every image $\varphi_{\sigma}(\sigma)$, $\sigma \in X$ is a general polyhedron, not       contained in a proper affine subspace of $\RR^{n_{\sigma}}$,
    \item for every pair $\sigma, \sigma' \in X$
    the composition $\varphi_{\sigma} \circ \varphi_{\sigma'}^{-1}$ is integer
    affine-linear on $\varphi_{\sigma'}(\sigma \cap \sigma')$,
    \item $|X|=\bigcup\limits_{\sigma \in X}^{\mbox{\Large .}}
    \varphi_{\sigma}^{-1}(\varphi_{\sigma}(\sigma)^{\circ})$, where
    $\varphi_{\sigma}(\sigma)^{\circ}$ denotes the interior of
    $\varphi_{\sigma}(\sigma)$ in $\RR^{n_{\sigma}}$.
\end{enumerate} 
\end{definition}

\begin{definition}[General polyhedral complexes] \label{defn-polycomplex}
  A \df{(general) polyhedral complex} is a (general) polyhedral precomplex $(|X|,X,\{\varphi_{\sigma}|\sigma \in X\})$ such that for every polyhedron $\sigma \in X$ we are given an open fan $F_\sigma$ (denoted as well by $F_{\sigma}^X$ to underline that it belongs to the complex $X$ ) in some $\RR^{n^*_\sigma}$ and a homeomorphism
  $$\Phi_{\sigma}: S_{\sigma} = \bigcup_{\sigma' \in X, \sigma' \supseteq \sigma}
  (\sigma')^{ri} \stackrel{\sim}{\longrightarrow} |F_{\sigma}|$$
  satisfying:
  \begin{enumerate}
    \item for all $\sigma' \in X, \sigma' \supseteq \sigma$ one has $\Phi_{\sigma}(\sigma' \cap S_{\sigma}) \in F_{\sigma}$
    and $\Phi_{\sigma}$ is compatible with the $\ZZ$-linear structure
    on $\sigma'$, i.e. $\Phi_{\sigma} \circ \varphi_{\sigma'}^{-1}$
    and $\varphi_{\sigma'} \circ \Phi_{\sigma}^{-1}$ are integer affine
    linear on $\varphi_{\sigma'}(\sigma' \cap S_{\sigma})$, resp. $\Phi_{\sigma}(\sigma' \cap S_{\sigma})$,
    \item for every pair $\sigma, \tau \in X, $ there is an integer affine linear map $A_{\sigma,\tau}$ such that the following diagram commutes:
    $$\begin{xy}
      \xymatrix{ S{_\sigma} \cap S_{\tau} \ar[d]^{\sim}_{\Phi_{\sigma}} \ar[r]_{\sim}^{\Phi_{\tau}} &
      \Phi_{\tau}(S_{\sigma} \cap S_{\tau})\\
      \Phi_{\sigma}(S_{\sigma} \cap S_{\tau}) \ar[ur]_{A_{\sigma,\tau}} }
    \end{xy}.$$
  \end{enumerate}
  
  For simplicity we usually drop the embeddings $\varphi_{\sigma}$
  or the maps $\Phi_{\sigma}$ in the notation and denote the polyhedral complex $(X,|X|,\{ \varphi_{\sigma} | \sigma \in X\}, \{\Phi_{\tau}| \tau \in X\})$ by $(X,|X|,\{ \varphi_{\sigma} | \sigma \in X\})$ or by $(X,|X|,\{ \varphi\}, \{\Phi_{\tau}| \tau \in X\})$  or by $(X,|X|)$ or just by $X$ if no confusion can occur.
  The subsets $\sigma \in X$ are called the \df{general polyhedra}
  or \df{faces of $(X,|X|)$}. For $\sigma \in X$ the open set $\sigma^{ri} =
  \varphi_{\sigma}^{-1}(\varphi_{\sigma}(\sigma)^{\circ})$ is called the
  \df{relative interior of $\sigma$}. The \df{dimension} of
  $(X,|X|)$ is the maximum of the dimensions of its general polyhedra.
  We call $(X,|X|)$ \df{pure-dimensional} if all its inclusion-maximal general
  polyhedra are of the same dimension. We denote by $X^{(n)}$ the set
  of polyhedra in $(X,|X|)$ of dimension $n$. Let $\tau, \sigma \in X$. As in the case of
  fans we write $\tau \leq \sigma$ (or $\tau < \sigma$) if $\tau \subseteq \sigma$
  (or $\tau \subsetneq \sigma$, respectively). By abuse of notation we identify $\sigma$ with $\varphi_{\sigma}(\sigma)$.
\medskip \\
  A general polyhedral complex $(X,|X|)$ of pure dimension $n$
  together with a map $\omega_X : X^{(n)} \rightarrow \QQ$ is called
  \df{weighted polyhedral complex} of dimension $n$, and
  $\omega_X(\sigma)$ is called the \df{weight} of the polyhedron $\sigma \in X^{(n)}$, if all $F_{\sigma}$ are weighted open fans and
\begin{itemize}
 \item $\omega_X(\sigma') = \omega_{F_\sigma} (\Phi_{\sigma}
    (\sigma' \cap S_{\sigma}))$ for every $\sigma' \in (X)^{(n)}$ with
    $\sigma' \supseteq \sigma$,
\end{itemize}
  The empty complex $\emptyset$ is a
  weighted polyhedral complex of every dimension.
  If $\left( (X,\right.$ $|X|),$ $\left. \omega_X\right)$ is a weighted polyhedral
  complex of dimension $n$, then put $$X^{*}= {\{ \tau \in X | \tau \subseteq \sigma \text{ for
  some } \sigma \in X^{(n)} \text { with } \omega_X(\sigma) \neq 0
  \}}, |X^{*}| = \bigcup_{\tau \in X^{*}} \tau \subseteq |X|.$$
  Note that $\left( (X^{*},|X^{*}|), \omega_X|_{(X^{*})^{(n)}} \right)$
  is again a weighted polyhedral complex of dimension $n$. This complex is
  called the \df{non-zero part} of $\left( (X,|X|), \omega_X \right)$.
  We call a weighted polyhedral complex $\left( (X,|X|), \omega_X \right)$
  \df{reduced} if $\left( (X,|X|), \omega_X \right)=\left( (X^{*},|X^{*}|), \omega_{X^{*}} \right)$.
\end{definition}

\begin{definition}[Subcomplexes and refinements]
 Let $(X,$ $|X|,$ $\{ \varphi_{\sigma} | \sigma \in X \})$ and $(Y,$ $|Y|,$ $\{ \psi_{\tau} | \tau \in Y\})$ be two polyhedral complexes. We call $X$ a {\it subcomplex} of $Y$ if 
\begin{enumerate}
 \item $|X|\subseteq |Y|$,
 \item for every $\sigma$ in $X$ there exists a $\tau \in Y$ with $\sigma \subseteq \tau$, 
 \item for a pair $\sigma$ and $\tau$ from (b) the maps $\varphi_{\sigma}\circ \psi_{\tau}^{-1}$ and $\psi_{\tau} \circ \varphi_{\sigma}^{-1}$ are integer affine linear on $\psi_{\tau}(\sigma)$, resp. $\varphi_{\sigma}(\sigma)$.
\end{enumerate}
We write $(X,|X|)<(Y,|Y|)$ in this case, and define a map $C_{X,Y}:X\rightarrow Y$ that maps a cone in $X$ to the inclusion-minimal cone in $Y$ containing it.\\
We call a polyhedral complex $(X,|X|)$ a refinement of $(Y,|Y|)$, if
\begin{enumerate}
 \item $(X,|X|)<(Y,|Y|)$
 \item $|X|=|Y|$
\end{enumerate}
We call a weighted polyhedral complex $(X,|X|)$ a refinement of a weighted polyhedral complex $(Y,|Y|)$ if in addition the following condition holds:
\begin{itemize}
\item $\omega_X(\sigma) = \omega_Y(C_{X^{*},Y^{*}}(\sigma))$ for all
    $\sigma \in (X^{*})^{(\dim(X))}$.
\end{itemize}
\end{definition}

\begin{definition}[Morphism of general polyhedral complexes]
\label{def-morphgeneral}
 Let $X$ and $Y$ be two (general) polyhedral complexes. A \df{morphism of general polyhedral complexes} $f:X \rightarrow Y$ is a continuous map $f:|X|\rightarrow |Y|$ with the following properties: There exist refinements $(X',|X'|,\{\varphi\},\{ \Phi_{\sigma} | \sigma \in X' \})$ of $X$ and $(Y',|Y'|,\{\psi\},\{ \Psi_{\tau} | \tau \in Y')$ of $Y$ such that
\begin{enumerate}
    \item \label{gpma} for every general polyhedron $\sigma \in X'$ there exists a general polyhedron $\widetilde{\sigma} \in Y'$
    with $f(\sigma) \subseteq \widetilde{\sigma}$,
    \item for every pair $\sigma, \widetilde{\sigma}$ from \ref{gpma} the map
    $\Psi_{\widetilde{\sigma}} \circ f \circ \Phi_{\sigma}^{-1}: |F_{\sigma}^{X'}| \rightarrow |F_{\widetilde{\sigma}}^{Y'}|$
    induces a morphism of fans $\widetilde{F}_{\sigma}^{X'} \rightarrow
    \widetilde{F}_{\widetilde{\sigma}}^{Y'}$ , where $\widetilde{F}_{\sigma}^{X'}$ and
    $\widetilde{F}_{\widetilde{\sigma}}^{Y'}$ are the general fans given in Definition \ref{defn-openfan}.
 \end{enumerate}

A morphism of \df{weighted polyhedral complexes} is a morphism of polyhedral complexes (i.e. there are no conditions on the weights). If $X=Y$ and if there exists a morphism $g:X \rightarrow Y $ such that $g\circ f=f \circ g = id_X$ we call $f$ an {\it automorphism} of $X$.

\end{definition}

\begin{definition}[Orbit space]\label{arbitraryorbitspace}
Let $X$ be a polyhedral complex and $G$ a group acting on $|X|$ such that each $g \in G$ induces an automorphism on $X$.
We denote the induced map of an element $g \in G$ on $X$ by $g(.)$ and the induced homeomorphism on $|X|$ by $g\{.\}$. 
We denote by $X/G$ the set of $G-$orbits of $X$ and call $X/G$ an {\it orbit space}.
\end{definition}

\begin{example}
 The following example shows the topological space of an orbit space with trivial group $G$ and the open fans $F_{\sigma}$ for all $\sigma$. The group $G$ is trivial and thus the orbit space is the same as the polyhedral complex (i.e. $X=X/G$).
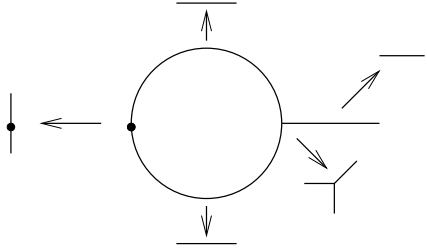
$$
\begin{minipage}{\linewidth}
\begin{center}
\begin{picture}(0,0)%
\includegraphics{quot.pstex}%
\end{picture}%
\setlength{\unitlength}{4144sp}%
\begingroup\makeatletter\ifx\SetFigFontNFSS\undefined%
\gdef\SetFigFontNFSS#1#2#3#4#5{%
  \reset@font\fontsize{#1}{#2pt}%
  \fontfamily{#3}\fontseries{#4}\fontshape{#5}%
  \selectfont}%
\fi\endgroup%
\begin{picture}(2521,1464)(3297,-3493)
\end{picture}%
\end{center}
\end{minipage}
$$
Let us now take the same polyhedral complex, with the ray lying on the $x$-axis. For $G$ take the group with two elements, generated by the map $f:\RR^2 \rightarrow \RR^2, y\mapsto -y$. The topological picture of the orbit space is the following:
$$
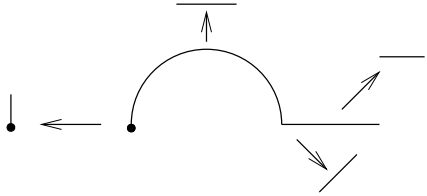
\begin{minipage}{\linewidth}
\begin{center}
\begin{picture}(0,0)%
\includegraphics{quot2.pstex}%
\end{picture}%
\setlength{\unitlength}{4144sp}%
\begingroup\makeatletter\ifx\SetFigFontNFSS\undefined%
\gdef\SetFigFontNFSS#1#2#3#4#5{%
  \reset@font\fontsize{#1}{#2pt}%
  \fontfamily{#3}\fontseries{#4}\fontshape{#5}%
  \selectfont}%
\fi\endgroup%
\begin{picture}(2522,1149)(3296,-3178)
\end{picture}%
\end{center}
\end{minipage}
$$
\end{example}

\begin{definition}[Weighted orbit space]
Let $\left( X,\omega_X \right)$ be a weighted polyhedral complex of dimension $n$, and $G$ a group acting on $X$. If $X/G$ is an orbit space such that
\begin{itemize}
\item for any $ g \in G$ and for any $\sigma \in X^{\left( n \right)}$, one has $ \omega_X \left( \sigma\right)= \omega_X \left( g(\sigma) \right)$,
\end{itemize}
we call $X/G$ a \df{weighted orbit space}.
The classes $[\sigma] \in X/G$, given by the orbits of $G$, are called {\it weighted classes}. 
\end{definition}

\begin{definition}[Stabilizer, $G_{\tau}-$orbit of $\sigma$]
Let $X$ and $G$ be as above and $\tau, \sigma \in X$. We call $G_{\tau}=\{ g \in G| g\{x\}=x\ \text{for any } x\in \tau\}$ the \df{stabilizer} of $\tau$. We define $X_{\sigma/\tau}=\{ g(\sigma)| g\in G_{\tau}\}$ to be the \df{$G_{\tau}-$orbit of $\sigma$}.
\medskip \\
The weight function on the weighted classes of $X/G$ is denoted by $[\omega]$ and defined by $[\omega]([\sigma])= \omega(\sigma)/|G_{\sigma}|$, for all $[\sigma]\in X/G$.
\end{definition}

\begin{remark}
We could define a weighted orbit space as well by giving an orbit space and a weight for each class instead of defining the weights of the orbit space by the weights of the complex and the group action.
\end{remark}

\begin{definition}[Suborbit space]
Let $X/G$ be an orbit space. An orbit space $Y/G$ is called a \df {suborbit space} of $X/G$ (notation: $ Y/G \subset X/G $) if each general polyhedron of $Y$ is contained in a general polyhedron of $X$ and each element of $G$ acts on the faces of $Y$ in the same way as for $X$ (i.e. for all $g\in G$, $\sigma \in Y$ we have $g_{|Y|}\{x\}=g_{|X|}\{x\}$ for $x\in \sigma$). In this case we denote by $ C_{Y,X}: Y \to X $ the map which sends a general polyhedron $\sigma \in Y $ to the (unique) inclusion-minimal general polyhedron of $X$ that contains $ \sigma $. Note that for a suborbit space $ Y/G \subset X/G $ we obviously have $ |Y| \subset |X| $ and $ \dim C_{Y,X} (\sigma) \ge \dim \sigma $ for all $ \sigma \in Y $.
\end{definition}

\begin{definition}[Refinements]
 Let $((Y,$ $|Y|),$ $\omega_Y)/G$ and $((X,$ $|X|),$ $\omega_X)$ $/G$ be two weighted orbit spaces. We call $((Y,|Y|),\omega_Y)/G$ a \df{refinement} of $((X,|X|),\omega_X)/G$, if 
\begin{enumerate}
 \item [(a)] $((Y,|Y|),\omega_Y)/G \subset ((X,|X|),\omega_X)/G$,
 \item [(b)] $|Y^*|=|X^*|$,
 \item [(c)] $\omega_Y(\sigma)=\omega_X(C_{Y,X}(\sigma))$ for all $\sigma \in (Y^*)^{(\dim (Y))}$,
 \item[(d)] each $\sigma \in Y$ is closed in $|X|$.
\end{enumerate}
We say that two weighted orbit spaces $((X,|X|),\omega_X)/G$ and $((Y,|Y|),\omega_Y)/G$ are equivalent (notation: $((X,|X|),\omega_X)/G \cong ((Y,|Y|),\omega_Y)/G$) if they have a common refinement. 
\end{definition}

\begin{definition}[Global orbit space]
 Let $F$ be a set of orbit spaces and $E$ a set of embeddings $ \phi_{X,Y,\sigma}:\sigma^{\circ} \rightarrow Y$, given by affine linear maps, of the interior of a polyhedron $\sigma\in X$, with $X$, $Y \in F$. After refinement of $Y$ there exists a cone $\tilde{\sigma}$ in $Y$, such that $\tilde{\sigma}^{\circ}=\phi_{X,Y,\sigma}(\sigma^{\circ})$. Since $\sigma \subset \RR^n$ and $\varphi_{\tilde{\sigma}}\circ\phi_{X,Y,\sigma}$ is an affine linear map there exists a continuation $\widetilde{\varphi_{\tilde{\sigma}}\circ\phi_{X,Y,\sigma}}$ of $\varphi_{\tilde{\sigma}}\circ\phi_{X,Y,\sigma}$ on $\sigma$.
If $\widetilde{\varphi_{\tilde{\sigma}}\circ\phi_{X,Y,\sigma}} \cap \varphi_{\tilde{\sigma}}(\tilde{\sigma})= \varphi_{\tilde{\sigma}}(\tilde{\sigma}^{\circ})$, we glue the orbit spaces along these maps. The resulting topological space together with $F$ and $E$ is called \df{global orbit space}.
\end{definition}

\begin{remark}
 The global orbit space is a topological space which locally is an orbit space. In the same way one could define a weighted and later on a global tropical orbit space. Perhaps one would prefer to call the orbit space local orbit space, and the global orbit space only orbit space, but since all our objects will have a global group operation we keep these names. For weighted global orbit spaces one would need the condition that the weights of the glued cones coincide.
\end{remark}

\begin{definition}[Tropical orbit space]\label{troporbitspace}
 Let $\left( X,\omega_X \right)/G$ be a weighted orbit space with finitely many different classes and $|G_{\sigma}|<\infty $ for any $\sigma \in X^{(n)}$. If 
for any $\tau \in X^{(n-1)}$, one has $ \# \{ \sigma >\tau \}$ $< \infty$ and there exists $\lambda_{\sigma/\tau}$ $\geq 0\ $ for any $\sigma$ $> \tau $ such that $\sum_{\tilde{\sigma} >\tau, \tilde{\sigma}\in X_{\sigma/\tau}}$ $\lambda_{\tilde{\sigma}/\tau}$ $=1$ 
and $\sum_{\sigma>\tau}$ $\lambda_{\sigma/\tau}$ $[\omega]_X([\sigma])(u_{\sigma /\tau})$ $\in V_{\tau}$, then $X/G$ is called a {\it tropical orbit space}.
\end{definition}

\begin{proposition}\label{equiv}
 Let $\left( X,\omega_X \right)$ be a weighted fan in $V$ and $G \subset Gl(V)$ such that $X/G$ is a weighted orbit space. If $G$ is finite and all $general\ cones$ in $X$ are $closed\ cones$, then $\left( X,\omega_X \right)$ is a tropical fan if and only if $X/G$ is a tropical orbit space.
\end{proposition}

\begin{proof}
 $"\Rightarrow"$: Put $n=\dim(X)$ and let $\tau \in X^{(n-1)}$ and $\sigma>\tau$. Then we define $\lambda_{\sigma/\tau}$ $=\frac{|\{g\in G_{\tau},\text{ such that } g(\sigma)=\sigma\}|}{|G_{\tau}|}=$ $\frac{|G_{\sigma}|}{|G_{\tau}|}=$ $\frac{1}{|X_{\sigma/\tau}|}$.
Thus, for any $\tau \in X^{(n-1)}$ one has $\# \{ \sigma >\tau \} < \infty,$ and for any $\sigma > \tau$ one has $\lambda_{\sigma/\tau} \geq 0\ $  and $ \sum_{\tilde{\sigma} >\tau, \tilde{\sigma}\in X_{\sigma/\tau}}\lambda_{\tilde{\sigma}/\tau}=1$.\\
Furthermore, $\sum_{\sigma>\tau}\frac{1}{|G_{\tau}|} \omega_X(\sigma)(v_{\sigma /\tau})=t \in V_{\tau}$, because $\left( X,\omega_X \right)$ is a tropical fan. Thus, we have $\sum_{\sigma>\tau}\frac{|G_{\sigma}|}{|G_{\tau}|} [\omega_X]([\sigma])(v_{\sigma /\tau})=\sum_{\sigma>\tau}\frac{1}{|G_{\tau}|} \omega_X(\sigma)(v_{\sigma /\tau})=t \in V_{\tau}$.

$"\Leftarrow"$: Let $X/G$ be a tropical orbit space. Thus, there exists $\lambda_{\sigma/\tau}$ with $\sigma > \tau$ and $\tau \in X^{(n-1)}$ such that $\sum_{\sigma>\tau}\lambda_{\sigma/\tau} [\omega]_X([\sigma])(u_{\sigma /\tau})=t \in V_{\tau}$. Therefore, because of the linearity of $g \in G_{\tau}$, we get:
\begin{eqnarray*}
 |G_{\tau}|\cdot t&=& \sum_{g\in G_{\tau}}g(t)\\ &=& \sum_{g\in G_{\tau}} g(\sum_{\sigma>\tau}\lambda_{\sigma/\tau} [\omega]_X([\sigma])(u_{\sigma /\tau}))\\&=& \sum_{g\in G_{\tau}} \sum_{\sigma>\tau}\lambda_{\sigma/\tau} [\omega]_X([\sigma])(g(u_{\sigma /\tau}))\\&=&
 \sum_{\sigma>\tau}|G_{\sigma}|\cdot [\omega]_X([\sigma])(u_{\sigma /\tau})\\&=&
 \sum_{\sigma>\tau}\omega_X(\sigma)(u_{\sigma /\tau}).
\end{eqnarray*}
\end{proof}

\begin{example}\label{line}
 The following picture is an example of a tropical fan $X$ and a tropical orbit space $X/G$ with this fan as underlying polyhedral complex.
Let $X$ be the standard tropical line with its vertex at the origin, given by the directions $\left( \begin{array}{c} -1 \\ 0  \end{array} \right),$ $ \left( \begin{array}{c} 0\\ -1 \end{array}\right)$ and $\left( \begin{array}{c} 1\\ 1 \end{array} \right)$, and all the weights are equal to one. The group $G$ consists of two elements and is generated by the matrix $\left( \begin{array}{cc} 0 & 1\\ 1 & 0 \end{array} \right)$.
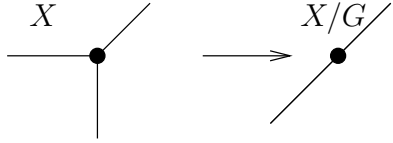
$$
\begin{minipage}{\linewidth}
\begin{center}
\input{tropical.pstex_t}
\end{center}
\end{minipage}
$$
The balancing condition for the fan is 
$$\left( \begin{array}{c} -1 \\ 0  \end{array} \right)+ \left( \begin{array}{c} 0\\ -1 \end{array}\right)+\left( \begin{array}{c} 1\\ 1 \end{array} \right)= \left( \begin{array}{c} 0\\ 0 \end{array}\right)$$
and for the orbit space
$$\frac{1}{2} \cdot \left( \begin{array}{c} -1 \\ 0  \end{array} \right)+ \frac{1}{2} \cdot \left( \begin{array}{c} 0\\ -1 \end{array}\right)+ \frac{1}{2} \cdot \left( \begin{array}{c} 1\\ 1 \end{array} \right)= \left( \begin{array}{c} 0\\ 0 \end{array} \right),$$
where the first two $(1/2)$'s come from the splitting of $1$, and the third $1/2$ comes from the invariance of the last vector under $G$.
\end{example}

\begin{corollary}\label{fancondition}
The balancing condition for tropical orbit spaces can be checked by checking the balancing condition of the underlying weighted complex. 
\end{corollary}
\begin{proof}
  For tropical orbit spaces with infinite group $G$ there are only finitely many facets around a codim-$1$ face. Thus, as in the proof of proposition \ref{equiv} the balancing condition can be checked on the polyhedral complex as well (without group action).
\end{proof}

\begin{example}
 To show that there are tropical orbit spaces which do not come from a tropical fan we consider the following orbit space.
Let $|X|$ be the topological space $\{(x,y)\in \RR^2| y>0\}$, and let $X$ be the set of cones spanned by the vectors $\binom{x}{1}$ and $\binom{x+1}{1}$ for $x \in \ZZ$. If we define all weights to be one and $G=<\left(\begin{array}{cc} 1&1\\ 0&1\end{array}\right)>$, we get the following tropical orbit space $X/G$:

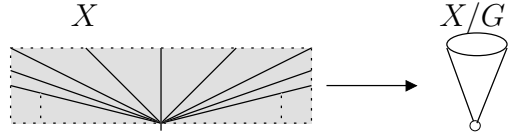
$$
\begin{minipage}{\linewidth}
\begin{center}
\input{tropicalstack.pstex_t}
\end{center}
\end{minipage}
$$
It is easily be seen, that $X/G$ is a tropical orbit space (see definition \ref{troporbitspace}), while $X$ has infinitely many cones and thus is not a fan.
\end{example}

\begin{definition}[Morphism of orbit spaces]
\label{def-morphism}
 Let $(X,$ $|X|,$ $\{\varphi\},$ $\{ \Phi_{\sigma} | \sigma \in X \})/$ $G$ and $(Y,$ $|Y|,$  $\{\psi\},$ $\{  \Psi_{\tau} | \tau \in Y)$ $/H$ be two orbit spaces. A \df{morphism of orbit spaces} $f:X/G \rightarrow Y/H$ is a pair $(f_1,f_2)$ consisting of a continuous map $f_1:|X|\rightarrow |Y|$ and a group morphism $f_2:G\rightarrow H$ with the following properties:
\begin{enumerate}
    \item \label{pma} for every general polyhedron $\sigma \in X$ there exists a general polyhedron $\widetilde{\sigma} \in Y$
    with $f_1(\sigma) \subseteq \widetilde{\sigma}$,
    \item for every pair $\sigma, \widetilde{\sigma}$ from \ref{pma} the map
    $\Psi_{\widetilde{\sigma}} \circ f_1 \circ \Phi_{\sigma}^{-1}: |F_{\sigma}^{X}| \rightarrow |F_{\widetilde{\sigma}}^{Y}|$
    induces a morphism of fans $\widetilde{F}_{\sigma}^{X} \rightarrow
    \widetilde{F}_{\widetilde{\sigma}}^{Y}$ , where $\widetilde{F}_{\sigma}^{X}$ and
    $\widetilde{F}_{\widetilde{\sigma}}^{Y}$ are the weighted general fans associated to $F_{\sigma}^{X}$ and $F_{\widetilde{\sigma}}^{Y}$,
    respectively (cf. definition \ref{defn-openfan}),
    \item there exists a refinement of $X$ such that for any $\sigma,$ $ \tilde{\sigma}$ $ \in X$ with $\dim(f_1(\sigma)$ $\cap f_1(\tilde{\sigma}))$ $=\dim(f_1(\sigma))$ $=\dim(f_1(\tilde{\sigma}))$, one has $f_1(\sigma)=f_1(\tilde{\sigma})$,
    \item $f_1(g(\sigma))$ $=f_2(g)(f_1(\sigma))$ for all $g\in G$ and $\sigma \in X$.
 \end{enumerate}

A morphism of \df{weighted orbit spaces} is a morphism of orbit spaces (i.e. there are no conditions on the weights).

\end{definition}

\begin{explanation}
Asking a morphism to fulfill conditions $a,b$ and $d$ is obvious, but to ask for condition $c$ is not. Thus, let us consider an example where condition $c$ is not fulfilled.\\
Let us consider the map $f$, given by the projection of two intervals on a third one (see the following picture). We take $G$ and $H$ to be trivial, thus $X/G=X$ and $Y/H=Y$, where $X$ is the disjoint union of two open intervals of different length and $Y$ is one open interval with the same length as the longest interval of $X$.
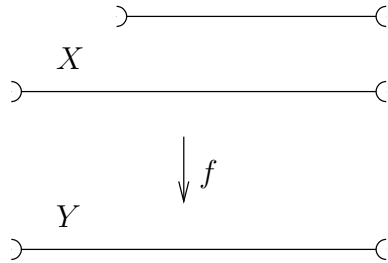
 $$
\begin{minipage}{\linewidth}
\begin{center}
\input{conditionc.pstex_t}
\end{center}
\end{minipage}
$$
After any possible refinement the facet $\sigma$, which is the most left in the upper interval of $X$, is open on the left side, but will be mapped on a left closed facet $\tau$. We call $\tilde{\sigma}$ the intersection of the preimage of $\tau$ with the longest interval of $X$ . Then $f_1(\sigma)\cap f_1(\tilde{\sigma})$ is a line segment as well as $f_1(\sigma)$ and $f_1(\tilde{\sigma})$, but the images are not the same which contradicts $c$. Thus $f$ is not a morphism.
\end{explanation}

\begin{example}
 If we take the tropical orbit space $X/G$ from Example \ref{line}, then the canonical map to the diagonal line in $\RR^2$ is a morphism of orbit spaces. But the homeomorphism which goes in the opposite direction is not a morphism, because locally at the origin it can not be expressed by a linear map.
\end{example}

\begin{remark}
The reason we ask condition $c$ to be fulfilled is to define images of the polyhedra later on. Thus, after refinement, each polyhedron should map to one polyhedron and the image of the polyhedral complex should be a polyhedral complex as well. In particular condition $a$ of Definition \ref{defn-prepolycomplex} has to be fullfiled. Therefore, different images of polyhedra should intersect in lower dimension than the maximal dimension of them.
Or in other words, $c$ ensures $a$ in Definition \ref{defn-prepolycomplex}.
\end{remark}

\begin{construction}
 As in the case of fans (Construction 2.24 \cite{GKM}) we can define the image orbit space. Let $X/G$ be a purely $n$-dimensional orbit space, and let $Y/H$ be any orbit space. For any morphism $X/G \rightarrow Y/H$ consider the following set:
\begin{eqnarray*}
Z & = &\{f(\sigma),\sigma \text{ is contained in a cone } \tilde{\sigma} \text{ of } X^{(n)} \text{ with $f$ is injective on } \tilde{\sigma}\}
\end{eqnarray*}

Note, that $Z$ is in general not a polyhedral complex. It satisfies all conditions of Definition \ref{defn-prepolycomplex} and Definition \ref{defn-polycomplex} except possibly $(d)$ of Definition \ref{defn-prepolycomplex} (since there might be overlaps of some regions). However, we can choose a proper refinement to turn $Z$ into a polyhedral complex. Thus, if we denote the weighted polyhedral complex defined by all representatives of all classes $[\sigma]$ with $\sigma \in Z$ by $H\circ Z$, we get an orbit space $H\circ Z/H$.

If moreover $X/G$ is a weighted orbit space we turn $f(X/G)$ into a weighted orbit space. After choosing a refinement for $X$ and $Y$ such that ${f(\sigma)}$ is a cone in $Y$ for each $\sigma \in X$, we set
 \[ \omega_{f(X/G)} (\sigma') =
         \sum_{[\sigma] \in X/G^{(n)}:[f(\sigma)]=[{\sigma}']}
           \omega_{X} (\sigma) \cdot |\Lambda'_{[\sigma']}/f(\Lambda_{[\sigma]})| \]
for any $\sigma' \in (H\circ Z)^{(n)}$.
\end{construction}

\begin {proposition} \label {prop-image}
  Let $X/G$ be an $n$-dimensional tropical orbit space, 
  $Y/H$ an orbit space, and $ f: X/G \to Y/H $ a
  morphism. Then $ f(X/G) $ is an $n-$dimensional tropical orbit space (provided that $f(X/G)$ is not empty).
\end {proposition}

\begin{proof}
  By construction, $f(X/G)$ is an $n-$dimensional weighted orbit space. Thus we have to prove only the balancing condition. The proof works in the same way as for fans in \cite{GKM} (Notice that by Corollary \ref{fancondition} the balancing condition can be checked without taking into account the group operation).
\end{proof}

\begin{definition}[Irreducible tropical orbit space]
 Let $X/G$ be a tropical orbit space of dimension $n$. We call $X/G$ \df{irreducible} if for any  refinement $\tilde{X}/G$ of $X/G$ and any $Y/G\subset X/G, Y\neq \emptyset$ with $\dim (Y/G)=n$ the following holds: if for all $\sigma \in Y^{(n)}$ one has $\sigma \in \tilde{X}^{(n)}$, then $Y$ and $\tilde{X}$ are equal. (The equality holds on the level of orbit spaces, the weights can be different. In the case of different weights one has $\omega_X=\lambda \cdot \omega_Y$ for $\lambda\in \QQ \neq 0$). Equivalent to this definition is to say that $X/G$ is \df{irreducible}, if for any $Y/G\subset X/G, Y\neq \emptyset$ with $\dim (Y/G)=n$ and $Y$ is closed in $X$ one has $Y=X$.
\end{definition}

\begin {corollary} \label {cor-image}
  Let $X/G$ and $Y/H$ be tropical orbit spaces of the same dimension $n$ in $ V = \Lambda
  \otimes \RR $ and $ V' = \Lambda' \otimes \RR $, respectively, and let $ f: X/G
  \to Y/H $ be a morphism. Assume that $Y/H$ is irreducible and $f(X/G)=Y/H$ as topological spaces. Then there is an orbit space $ Y_0/H $ in $V'$ of dimension smaller than $n$ with $ |Y_0| \subset |Y| $ such that
  \begin {enumerate}
  \item \label {cor-image-a}
    each point $ Q \in |Y| \backslash |Y_0| $ lies in the interior of a
    cone $ \sigma_Q' \in Y $ of dimension $n$;
  \item \label {cor-image-b}
    each point $ P \in f^{-1} (|Y| \backslash |Y_0|) $ lies in the interior
    of a cone $ \sigma_P \in X $ of dimension $n$;
  \item \label {cor-image-c}
    for $ Q \in |Y| \backslash |Y_0| $ the sum
      \[ \sum_{[P], P \in |X|: f([P])=[Q]} \mult_{[P]} f \]
    does not depend on $Q$, where the multiplicity $ \mult_{[P]} f $ of $f$ at $[P]$
    is defined to be
      \[ \mult_{[P]} f := \frac {\omega_{X/G}(\sigma_{[P]})}{\omega_{Y/H}(\sigma'_{[Q]})}
           \cdot |\Lambda'_{\sigma'_{[Q]}}/f(\Lambda_{\sigma_{[P]}})|. \]
  \end {enumerate}
\end {corollary}

\begin{proof}
If we can show that $f(X/G)= \lambda Y/H$ (i.e. the image of $X/G$ is $Y/G$ and the weights differ by the multiplication of $\lambda \in \QQ $) the proof works as in \cite{GKM} for fans.\\
By assumption we have, that $f(X/G)=Y/H$, as orbit spaces (without weights). Further, by Proposition \ref{prop-image}, $f(X/G)$ is a tropical orbit space. Because of irreducibility we have $f(X/G)=\lambda Y/H$ as tropical orbit spaces.
\end{proof}

In contrast to the case of fans we need in the Corollary the assumption $f(X/G)=Y/H$. This is due to the fact, that we use non-closed polyhedra. Let us see what happens if we do not assume the above equality.

\begin{example}
Let $G$ be the trivial group and $X\subset \RR$ and $Y\subset \RR$ be open intervals of weight one with $X\subsetneqq Y$. Let $f:X \hookrightarrow Y$ be the inclusion.
 $$
\begin{minipage}{\linewidth}
\begin{center}
\input{openintervals.pstex_t}
\end{center}
\end{minipage}
$$
Then, all conditions of the corollary but the equality are fulfilled and the corollary does not hold.
\end{example}

\begin{definition}[Rational function]
 Let $Y/G$ be a tropical orbit space. We define a {\it rational function} $\varphi$ on $Y/G$ to be a continuous function $\varphi : |Y| \rightarrow \RR$ such that there exists a refinement $(((X, |X|, \{m_\sigma\}_{\sigma \in X}), \omega_X), \{M_\sigma\}_{\sigma \in X})$ of $Y$ fullfiling: for each face $\sigma \in X$ the map $\varphi \circ m_\sigma^{-1}$ is locally integer affine-linear. Furthermore, we demand that $f \circ g=f,$ for all $g \in G$. (Remark: by refinements we can directly assume that $f$ is affine linear on each general cone.)
\end{definition}

\begin{definition}[Orbit space divisor]\label{tropicalorbitspace}
 Let $X/G$ be a tropical orbit space, and $\phi$ a rational function on $X/G$. We define a divisor of $\phi$ to be $div(\phi)$ $=\phi\cdot X/G$ $= [(\bigcup_{i=-1}^{k-1} X^{(i)},$ $\omega_\phi)]$ $/G$, 
 where $\omega_{\phi}$ is given as follows:
\begin{eqnarray*}
        \omega_\phi : X^{(k-1)} & \rightarrow & \QQ, \\
        \tau                                             & \mapsto     & \sum_{\twolines{\sigma \in X^{(k)}}{\tau < \sigma}} \phi_\sigma(\lambda_{\sigma/\tau}\omega(\sigma) v_{\sigma / \tau}) -
    \phi_\tau\Big( \sum_{\twolines{\sigma \in X^{(k)}}{\tau < \sigma}}\lambda_{\sigma/\tau} \omega(\sigma) v_{\sigma / \tau}\Big)
    \end{eqnarray*}
\end{definition}

\begin{remark}\label{indep} 
The following two remarks can be proved analogously to the proof of Proposition \ref{equiv}.
\vspace{-0.3 cm}
\begin{description}
 \item [1] The definition above is independent of the chosen $\lambda_{\sigma/\tau}$(i.e. if we have different sets of $\lambda$'s fulfilling the definition of a tropical orbit space, the divisor will be the same for both sets of $\lambda$'s).
 \item [2] If $|G_{\sigma}|<\infty$ for all $\sigma \in X^{(n-1)}$ and the number $|\{ \sigma >\tau \}| < \infty$ for all $\tau \in X^{(n-2)}$, then $\phi\cdot X$ is a tropical orbit space.
\end{description}
\end{remark}

\section{moduli spaces of elliptic tropical curves}
In this section we show that the moduli spaces of tropical curves of genus $1$ with $j$-invariant greater than $0$ have a structure of tropical orbit space.

\begin{definition}[$n$-marked abstract tropical curves]\label{markedcurve}
 An \df{abstract tropical curve} is a pair ($\overline{\Gamma},$ $\delta $) such that $\overline{\Gamma}$ is a connected graph, and $\Gamma=\overline{\Gamma}\backslash$ \{1-valent vertices\} has a complete inner metric $\delta$ (i.e. the edges adjacent to two vertices of $\Gamma$ are isometric to a segment, the edges adjacent to one vertex of $\Gamma$ are isometric to a ray and the edges adjacent to no vertex of $\Gamma$ are isometric to a line). The edges adjacent to no or to exactly one vertex of $\Gamma$ are called unbounded, the other edges are called bounded. The unbounded edges have length infinity. The bounded edges have a finite positive length. For simplicity we denote an abstract tropical curve by $\Gamma$. An \df{$n$-marked abstract tropical curve} is a tuple ($\Gamma,x_1,...,x_n$) formed by an abstract tropical curve $\Gamma$ and distinct unbounded edges $x_1,...,x_n$ of $\Gamma$ which are rays.
Two such marked tropical curves ($\Gamma,x_1,...,x_n$) and ($\widetilde{\Gamma},\tilde{x_1},...,\tilde{x_n}$) are called {\it isomorphic} (and will from now on be identified) if there exists an isometry from $\Gamma$ to $\widetilde{\Gamma}$, mapping $x_i$ to $\tilde{x_i}, i= 1, ..., n$ ( i.e. there exists a homeomorphism $\Gamma \rightarrow \widetilde{\Gamma}$ identifying $x_i$ and $\tilde{x_i}$ and such that the edges of $\Gamma$ are mapped to edges of $\widetilde{\Gamma}$ by an affine map of slope $\pm 1$.).
\end{definition}

For a more detailed definition of an abstract tropical curve see \cite{GM3} definition 2.2. The unbounded edges are called {\it leaves} as well.

\begin{remark}
 We can parameterize each edge $E$ of a curve $\Gamma$ by an interval $[0,l(E)]$ for bounded edges and by $[0,\infty)$ or $(-\infty,\infty)$ for unbounded edges, where $l(E)$ is the length of the edge. For the choice of the direction in the bounded case we choose which vertex of $E$ is parameterized by $0$. Such a parameterization is called {\it canonical}. We do not distinguish between the unbounded edge $x_i$ and the vertex adjacent to it and call the vertex also $x_i$.
\end{remark}

\begin{definition}[$n$-marked abstract topical curves of genus $1$]
 We call an $n$-marked abstract tropical curve to be of {\it genus $1$} if the underlying graph has exactly one simple cycle. 
\end{definition}

As a tropical counterpart of the $j$-invariant, we take the length of the cycle as it was suggested in \cite{M3}, \cite{V} and \cite{KM}. Motivations for this choice can be found, for example, in \cite{KMM1}, \cite{KMM2} and \cite{Sp}.

\begin{definition}[j-invariant]
For an $n$-marked curve $\Gamma$ of genus $1$, the sum of the lengths of all edges forming the simple cycle is called the {\it j-invariant} of $\Gamma$.
\end{definition}

\begin{definition}[Combinatorial type]
 The {\it combinatorial type} of an abstract tropical curve ($\overline{\Gamma},$ $\delta $) is the graph $\overline{\Gamma}$. 
\end{definition}


\begin{remdef}\label{combinatorial}
 All curves given by Definition \ref{markedcurve} of the same combinatorial type or the combinatorial type one gets by contracting bounded edges of the graph of the combinatorial type can be embedded in a suitable $\RR^m$ by the lengths of the bounded edges and therefore this set of curves has a topological structure (called combinatorial cone). Thus, the set of all $n$-marked abstract tropical curves of genus $1$ with this induced topological structure on each combinatorial cone (the cones are glued together along faces representing the same curves) is a topological space.
\end{remdef}

\begin{definition}[abstract $\calM_{1,n}$]
The space $\calM_{1,n}$ is defined to be the topological space of all $n$-marked abstract tropical curves (modulo isomorphism) with the following properties:
\vspace{-0.5 cm}
\begin{enumerate}
 \item the curve has exactly $n$ leaves, 
 \item all vertices of the curves have valence at least $3$, and
 \item the genus of the curve is $1$.
\end{enumerate}
The topology of this space is the one defined in the previous remark and definition. 
\end{definition}

\begin{example}\label{m12}
The moduli space of $2$-marked abstract tropical curves of genus $1$ and the curves corresponding to the faces are given in the following picture:
$$
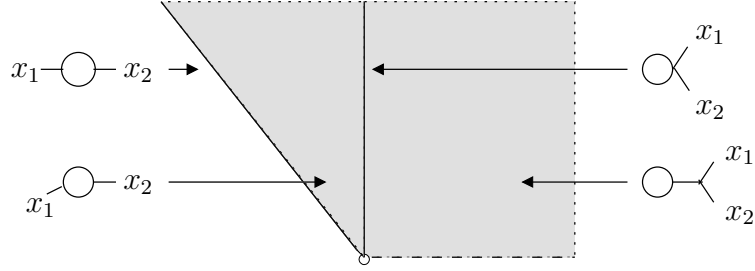
\begin{minipage}{\linewidth}
\begin{center}
\input{abstractmodulispace2.pstex_t}
\end{center}
\end{minipage}
$$
\end{example}

Now we construct a map from $\calM_{1,n}$ to a tropical orbit space in the following way. For each curve $C \in \calM_{1,n}$ let $a$ be an arbitrary point of the cycle of $C$. We define a new curve $\tilde{C}$ which we get by cutting $C$ along $a$ and inserting two leaves $A=x_{n+1}$ and $B=x_{n+2}$ at the resulting endpoints (if we cut along a vertex we have to decide if the edges adjacent to the vertex which are not in the cycle are adjacent to $A$ or to $B$). This curve is an $n+2$ marked curve (not of genus $1$) with up to $2$ two-valent vertices (at the ends $A$ and $B$).

\begin{figure}[h]
         \centerline{
          \input{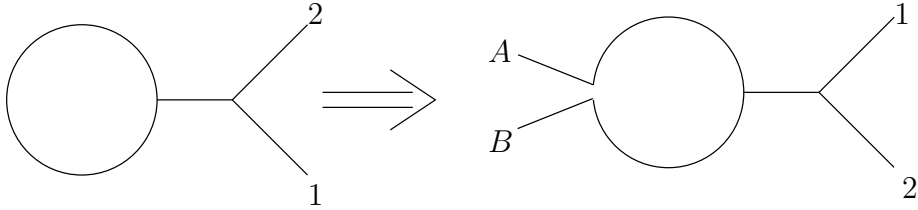}
             }
          \caption{Construction of an $n+2$-marked curve from an $n$-marked genus-$1$ curve.}
\end{figure}


Let $\mathcal{T}$ be the set of all subsets $S \subset \{1,\ldots,n+2\}$ with $|S|=2$. In order to embed $\calM_{1,n}$ into a quotient of $\RR^{\binom {n+2} 2}$ we consider the following map:
\begin{eqnarray*}
  \dist_n: \calM_{1,n} & \longrightarrow  & V_n/G_n \\
           (C,x_1,\ldots,x_n) & \longmapsto
             & [(\dist_\Gamma(x_i,x_j))_{\{i,j\}\in\mathcal{T}}]
\end{eqnarray*}
 
where $V_n,\ G_n,$ and $\dist_\Gamma(x_i,x_j)$ are defined as follows. We denote by $\dist_\Gamma(x_i,x_j)$ the distance between $x_i$ and $x_j$ (that is the sum of the lengths of all edges in the unique path from $x_i$ to $x_j$) in $\tilde{C}$, where $x_{n+1}=A$ and $x_{n+2}=B$.\\
Let $b \in \RR^t$. If we denote by $b_i$, $0<i\leq t$, the $i$th entry of $b$, then the vector space $V_n$ is isomorphic to $\RR^{\binom{n+2}{2}-n-1}$ and is given by $V_n=\RR^{\binom{n+2}{2}}/ (\Phi_n(\RR^n)+<s>)$ where

$$\begin{array}{lllll}
 \Phi_n: \RR^n & \longrightarrow & \RR^{n+2} &\longrightarrow & \RR^{\binom{n+2}{2}}\\
          b & \longmapsto & (b,0,0)=\tilde{b}& \longmapsto & (\tilde{b_i}+\tilde{b_j})_{{\{i,j\}\in\mathcal{T}}},
\end{array}$$

and $s\in \RR^{\binom{n+2}{2}}$ is a vector such that

$$
        s_{i,j} =
        \left\{ \begin{array}{lll}
            1   &   \text{if } i=n+1 \text{ or } j=n+1 \text{ and } i\neq n+2 \neq j, \\
            -1  &   \text{if } i=n+2 \text{ or } j=n+2 \text{ and } i\neq n+1 \neq j, \\
	    0   &   \text{otherwise.}
        \end{array} \right.
$$

The group $G_n$ is generated by the matrix $I$ and the matrices $M_p,\ p\in \{1,...,\ n\}$, where

$$
	I_{(i,j),(k,l)}=
	\left\{ \begin{array}{llll}
            1   &   \text{if } (\{i,j\},\{k,l\})=(\{m,n+1\},\{m,n+2\}),\ m\leq n,\\
	        &   \text{or } (\{i,j\},\{k,l\})=(\{m,n+2\},\{m,n+1\}),\ m\leq n,\\
               &    \text{or } \{ i,j\}= \{k,l\} \text{ and } i,j\notin\{n+1,n+2\},\\
		&   \text{or if } \{ i,j\}=\{n+1,n+2\}=\{k,l\},\\
	    0   &   \text{otherwise.}
        \end{array} \right.
$$ 

$$
	M_{p,(i,j),(k,l)}=
	\left\{ \begin{array}{llll}
            1   &   \text{if } \{i,j\}=\{k,l\}\\
            & \text{or } (\{i,j\},\{k,l\})=(\{p,n+2\},\{n+1,n+2\}), \\
	        &   \text{or } (\{i,j\},\{k,l\})=(\{p,j\},\{j,n+1\}),\ j\neq n+2,\\
                &   \text{or } (\{i,j\},\{k,l\})=(\{p,j\},\{p,n+2\}),\ j \neq n+2,\\
		&   \text{or } (\{i,j\},\{k,l\})=(\{p,j\},\{n+1,n+2\}),\\
		& \qquad n+1 \neq j\neq n+2,\\
	   -1   &   \text{ if } (\{i,j\},\{k,l\})=(\{p,n+1\},\{n+1,n+2\}), \\
		&   \text{or } (\{i,j\},\{k,l\})=(\{p,j\},\{j,n+2\}),\ j\neq n+1,\\
		&   \text{or } (\{i,j\},\{k,l\})=(\{p,j\},\{p,n+1\}),\ j \neq n+1,\\
	    0   &   \text{otherwise.} 
        \end{array} \right.
$$
The orbits of all elements of $<\Phi_n(\RR^n)>+<s>$ under $G_n$ are trivial and thus $V_n/G_n$ is well defined. By the following lemma, the definition of the map $\dist_n$ is well defined as well.

\begin{lemma}
 Let $\tilde{C}$ and $\tilde{C}^*$ be two curves resulting from two different cuts of a curve $C$. The images of $\tilde{C}$ and $\tilde{C}^*$ are the same in $V_n/G_n$.
\end{lemma}
\begin{proof}
Let us fix an orientation $o$ of the simple cycle in $C$ and let $\dist(\tilde{C})$ and $\dist(\tilde{C}^*)$ be the images under $\dist_\Gamma$ of $\tilde{C}$ and $\tilde{C}^*$. The orientation $o$ of the simple cycle in $C$ induces an orientation of the edges connecting $A$ and $B$ of $\tilde{C}$ and $\tilde{C}^*$. By applying the map $I$ to $\dist(\tilde{C})$ and $\dist(\tilde{C}^*)$ if necessary we can assume that the induced orientation goes from the $A$s to the $B$s. Let us denote by $\tilde{a},\tilde{A},\tilde{B}$ (resp. $\tilde{a}^*,\tilde{A}^*,\tilde{B}^*$) the cut and the inserted edges corresponding to curve $\tilde{C}$ (resp. $\tilde{C}^*$). We denote by $d$ the distance of $\tilde{B}$ to $\tilde{A}^*$ in the curve cut at $\tilde{a}$ and $\tilde{a}^*$. 
Let $L$ be the subset of marked points of the component containing $\overline{\tilde{B}\tilde{A}^*}$. Then the following equality holds:
$$
\dist(\tilde{C})=\prod_{p\in L}M_p\cdot \dist(\tilde{C}^*)+d\cdot s.
$$
\end{proof}

\begin{remark}
 The main idea in our definition comes from the rational case (see \cite{GKM}). After cutting the curve we get a new curve without cycles. Thus, the distance of two points in the new curve is well defined. 
Then, as in the rational case we have to mod out the image of $\Phi_n$. In addition we have to get rid of all the choices we made during the construction of $A$ and $B$. These choices can be expressed by the following three operations.
\begin{enumerate}
 \item The shift of the point $a$ on one edge of the cycle (which corresponds to the addition of an element of $<s>$).
 \item Interchanging $A$ and $B$, which corresponds to the matrix $I$.
 \item The point $a$ jumps over the vertex adjacent to an unbounded edge $p$. The matrix corresponding to this operation is $M_p$. If the point $a$ jumps over a bounded edge $E$, the matrix corresponding to this operation is the product of all matrices $M_i$ with $i$ is connected with $E$ by edges not intersecting the cycle.
\end{enumerate}
\end{remark}

To get a polyhedral complex we put
\begin{eqnarray*}
 \Psi_n: V_n & \longrightarrow & V_n/G_n\\
         x & \longmapsto & [x]
\end{eqnarray*}
and
$$X_n=\Psi_n^{-1} (\dist_n( \calM_{1,n})).\\ 
$$

As general polyhedrons we take the cones induced by the combinatorial cones in $\calM_{1,n}$, defined in Remark and Definition \ref{combinatorial}. Thus, $G_n$ is a group acting on $X_n$ and we can consider the quotient topology on the orbit space $X_n/G_n$ (see Definition \ref{arbitraryorbitspace}). To have a weighted orbit space we choose all weights to be equal to one. To show that the spaces $\calM_{1,n}$ have a structure of tropical orbit space, we have to show that $\calM_{1,n}$ and $X_n/G_n$ are homeomorphic and that $X_n/G_n$ fulfills the balancing condition.

\begin{proposition}
 Let $X_n, G_n \text{ and } \calM_{1,n}$ be as above. Then $S:\calM_{1,n}\longrightarrow X_n/G_n,$ $(C,x_1,\ldots,x_n)  \longmapsto [(\dist_\Gamma(x_i,x_j))]_{\{i,j\}\in\mathcal{T}}$ is a homeomorphism.
\end{proposition}

\begin{proof}
 Surjectivity is clear from the definition, and $S$ is a continuous closed map. Thus, it remains to show that $S$ is injective. To show this, we prove that out of each representative of an element $[x]$ in the target we can construct some numbers which are the same for each representative of $[x]$. If these numbers determine a unique preimage, the injectivity follows. For this we take the following numbers which are independent of the representative:


$j=x_{n+1,n+2}=$ length of the circle, 

$d_i=(x_{i,n+1}+x_{i,n+2}-j)/2=$ distance from $i$ to the circle,

$d_{i,k}=|(x_{i,n+1}+x_{k,n+2})-d_i-d_k-j|=$ distance of $i$ and $k$ on the circle. 

If there are more than three marked edges $i_1,...i_r$ with $d_{i_s,i_t}$ equals $0$ or $j$, than we have to determine the distances these edges have one to each other. But, since these distances do not depend on the cycle, the edges in $X_n$ encoding these distances are invariant under $G_n$. Thus, we can reconstruct these distances, by considering the projection (not necessarily orthogonal) of $[x]$ to the fixed part of the cone (and thus the fixed part of each representative) in which $[x]$ lies. 
Thus, all distances are given, injectivity follows and we are done.
\end{proof}

\begin{proposition}
 The weighted orbit space $X_n/G_n$ is a tropical orbit space.
\end{proposition}

\begin{proof}
To show the balancing condition we have to consider the codim-$1$ cones and the facets adjacent to them. If there is more than one vertex on the circle of a curve corresponding to a point on a facet, then the stabilizer is trivial and we are in the same case as for the $\calM_{0,n}$. If there is only one vertex on the cycle we have the stabilizer $\{I,1\}$, the identity and $I$ (see above). The curves corresponding to the points in the interior of the codim-$1$ face have exactly one $4$-valent vertex. This vertex can be adjacent to the circle or not. Let us consider these two cases separately. The second case is trivial (the stabilizers are the same for all three facets and the balancing condition is the same as for $\calM_{0,n}$), thus let us assume, that the $4$-valent vertex is at the circle. Qualitatively, the codim-$1$ face, which we call $\tau$, corresponds to a curve as in the following picture:
\begin{figure}[h]
         \centerline{
          \input{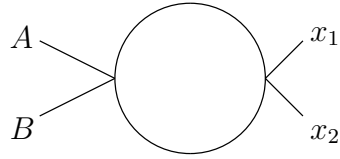}
             }
          \caption{A tropical curve with 4-valent vertex.}
\end{figure}

By assumption, there is only one vertex on the cycle. We only consider the case with two ends $x_1$ and $x_2$, because if we have a tree instead of $x_i$ the calculation is the same for each leaf of the tree.
To verify the balancing condition for tropical orbit spaces given in definition \ref{tropicalorbitspace}, we have to consider the three facets around the face $\tau$.
Let $\sigma_1$ (resp. $\sigma_2$) belong to the insertion of the edge with $A$ and $x_1$ (resp. $A$ and $x_2$) on the same side. Then, $\sigma_1$ and $\sigma_2$ lie in the same $G_{\tau}$-orbit. Thus, if we take the same notation as in the picture we get the following condition:\\
There exists $\lambda_1,\lambda_2 \geq 0,\ \lambda_{\sigma_1/ \tau}+\lambda_{\sigma_2/ \tau}=1$ such that
$$
\lambda_{\sigma_1/ \tau} \cdot \left( \begin{array}{c} 1 \\0 \\1 \\1 \\0 \\1 \end{array} \right)+
\lambda_{\sigma_2/ \tau} \cdot \left( \begin{array}{c} 1 \\1 \\0 \\0 \\1 \\1 \end{array} \right)+
\frac{1}{2}\cdot \left( \begin{array}{c} 0 \\1 \\1 \\1 \\1 \\0 \end{array} \right)
\begin{array}{c} d(x_1,x_2) \\ d(x_1,A) \\ d(x_1,B) \\ d(x_2,A) \\ d(x_2,B) \\ d(A,B) \end{array},\ \in V_{\tau}.
$$ 
This condition is fullfiled for $\lambda_{\sigma_1/ \tau}=\lambda_{\sigma_2/ \tau}= \frac{1}{2}$. Thus we have indeed a tropical orbit space.
\end{proof}

\begin{remark}
 In Example \ref{m12} we have seen the topological picture of the moduli space $\calM_{1,n}$.
 Unfortunately it is not possible to give a picture of the corresponding polyhedral complex since $X_2$ has infinitely many cones. But let us see a description of it. Let the vector entries be labeled as in the previous proof and let $C_1,C_2,C_3,C_4$ be the cones corresponding to the four different combinatorial cones in the picture of Example \ref{m12}, where $C_1$ is the left, $C_2$ the second left, $C_3$ the third left and $C_4$ the right combinatorial type. The group and represantantives of the cones $C_1,C_2,C_3,C_4$ (labeled by the same name) are the following:
$$
G=\left<\left( \begin{array}{cccccc} 1&0&0&0&0&0 \\0&0&1&0&0&0 \\0&1&0&0&0&0 \\0&0&0&0&1&0 \\0&0&0&1&0&0 \\0&0&0&0&0&1 \end{array} \right)\right.,\left. \left( \begin{array}{cccccc} 1&1&-1&-1&1&1 \\0&1&0&0&0&1 \\0&0&1&0&0&-1 \\0&0&0&1&0&0 \\0&0&0&0&1&0 \\0&0&0&0&0&1 \end{array} \right) \right>,\
$$
$$
C_1=\{a\cdot \left( \begin{array}{c} 1 \\1 \\1 \\2 \\0 \\2 \end{array}\right)|b>0\},\
C_2=\{a\cdot \left( \begin{array}{c} 1 \\1 \\1 \\2 \\0 \\2 \end{array}\right)+b\cdot \left( \begin{array}{c} 0 \\1 \\0 \\1 \\0 \\1 \end{array}\right)| a,b\in \RR_{\geq 0}, a+b>0\},
$$
$$ 
C_3=\{b\cdot \left( \begin{array}{c} 0 \\1 \\0 \\1 \\0 \\1 \end{array}\right)|b>0\},\
C_4=\{a\cdot \left( \begin{array}{c} 0 \\1 \\1 \\1 \\1 \\0 \end{array}\right)+b\cdot \left( \begin{array}{c} 0 \\1 \\0 \\1 \\0 \\1 \end{array}\right)| a,b\in \RR_{\geq 0}, b>0\}.
$$
All other cones of the underlying polyhedral complex are given by $g\{C_i\}$ for $g\in G$ and $i\in \{1,2,3,4\}$.
\end{remark}

Now we define a tropical orbit space corresponding to the parameterized genus-1 curves in $\RR^2$.

\begin{definition}[Tropical $ \widetilde{\calM}_{1,n}^{\lab}(\RR^r,\Delta) $]
    \label{def-m1nprd}
  A {\it parameterized labeled $n$-marked tropical curve of genus 1 in} $ \RR^r $ is a
  tuple $ (\Gamma,$ $x_1,$ $\dots,$ $x_N,$ $h) $, where $N \ge n$ is an integer,$
  (\Gamma,$ $x_1,$ $\dots,$ $x_N) $ is an abstract $N$-marked tropical curve of genus $1$, and $ h:
  \Gamma \to \RR^r $ is a continuous map satisfying the following conditions.
  \begin {enumerate}
  \item On each edge of $ \Gamma $ the map $h$ is of the form $ h(t) = a + t
    \cdot v $ for some $ a \in \RR^r $ and $ v \in \ZZ^r $. The integral
    vector $v$ occurring in this equation if we pick for $E$ the canonical
    parameterization starting at $V\in \partial E$ 
    is denoted $ v(E,V) $ and is called the \df {direction} of $E$ (at
    $V$). If $E$ is an unbounded edge and $V$ is its only boundary point we
    write $ v(E) $ instead of $ v(E,V) $ for simplicity.
  \item For every vertex $V$ of $ \Gamma $ we have the \df {balancing
    condition}
      \[ \sum_{E| V \in \partial E} v(E,V) = 0. \]
  \item $ v(x_i)=0 $ for $ i=1,\dots,n $ (i.e.\ each of the first $n$ leaves is
    contracted by $h$), whereas $ v(x_i) \neq 0 $ for $ i>n $ (i.e.\ the
    remaining $ N-n $ ends are ``non-contracted ends'').
  \end {enumerate}
  Two labeled $n$-marked tropical curves $ (\Gamma,x_1,\dots,x_N,h) $ and $
  (\tilde \Gamma, \tilde x_1,\dots,\tilde x_N,\tilde h)$ in $ \RR^r $ are
  called isomorphic (and will from now on be identified) if there is an
  isomorphism $\varphi: (\Gamma,x_1,\dots,x_N) \to (\tilde \Gamma,\tilde x_1,
  \dots,\tilde x_N) $ of the underlying abstract curves such that $ \tilde h
  \circ \varphi = h $.

 Let $m=N-n$. The \df {degree} of a labeled $n$-marked tropical curve $\Gamma$ of genus $1$ as above is defined
  to be the $m$-tuple $ \Delta = (v(x_{n+1}),\dots,v(x_N)) \in (\ZZ^r
  \backslash \{0\})^m$ of directions of its non-contracted ends. The \df
  {combinatorial type} of $\Gamma$ is given by the data of the combinatorial type of the
  underlying abstract marked tropical curve $ (\Gamma,x_1,\dots,x_N) $
  together with the directions of all its (bounded and unbounded)
  edges. From now on, the number $N$ will always be related to $n$
  and $ \Delta $ by $ N=n+\#\Delta $ and thus will denote the total number of
  (contracted or non-contracted) ends of an $n$-marked curve of genus $1$ in $ \RR^r $ of
  degree $ \Delta $.
  
  The space (of the isomorphism classes) of all labeled parameterized $n$-marked tropical curves of genus $1$ of a given degree
  $\Delta$ in $\RR^r$, such that all vertices have valence at least $3$ will be denoted $ \widetilde{\calM}_{1,n}^{\lab}(\mathbb{R}^r,
  \Delta) $. For the special choice
    \[ \Delta = (-e_0,\dots,-e_0\;\;,\dots,\;\;-e_r,\dots,-e_r) \]
  with $ e_0 = -e_1-\cdots-e_r $ and where each $ e_i $ occurs
  exactly $d$ times, we will also denote this space by $ \widetilde{\calM}_{1,n}^{
  \lab}(\mathbb{R}^r,d) $ and say that these curves have degree $d$.
\end{definition}

In the case of rational curves we can simply take  $\widetilde{\calM}_{0,n}^{\lab}(\mathbb{R}^r,
  \Delta)= \calM_{0,N}^{\lab} \times \RR^r $ because to build the moduli spaces of rational curves in $\RR^r$ it suffices to fix the coordinate of one of the marked ends (for example $x_1$). For the case of genus-1 curves the situation is more complicated. If we fix the combinatorial type of the curve, the cycle imposes some conditions on the lengths. In order to get a closed cycle in the image the direction vectors of the cycle edges multiplied by their lengths have to sum up to zero. Further we have to get rid of cells which are of higher dimension than expected. We will see that these conditions (closing of the cycle and getting rid of higher dimensional cells) can be expressed by some rational functions.

Let $\tilde{X}_{n,\Delta,r}^{\lab}= X_{N} \times \RR^r \times \ZZ^r$. We define $G_N^{\lab}$ to be as group the same as $G_N$, acting on $X_{N}$ as $G_N$ before, on $b\in \RR^r$ (that is the image of $x_1$) as identity and on $v\in \ZZ^r$ as follows:
$$
I(v)=-v,
M_{p}(v)=v-v(p).
$$
As topology on $\tilde{X}_{n,\Delta,r}^{\lab}$, we take the product topology of $X_N$, $\ZZ^r$ and $\RR^r$, where we consider $\ZZ^r$ with the discrete topology and $\RR^r$ with the standard Euclidean topology.

As mentioned above the direction vectors of the cycle multiplied by the lengths have to sum up to zero. To ensure that the sum is indeed $0$ we use divisors. Therefore, we need rational functions as defined in the second chapter. The purpose of these rational functions is to make sure that the $i$th coordinate of $A$ is mapped to the $i$th coordinate of $B$.

\begin{proposition}\label{rationalmap}
For all $0<i\leq r$, we have a function
\begin{eqnarray*}
 \phi_i: \tilde{X}_{n,\Delta,r}^{\lab} & \to &\RR\\
 (a _{\{1,2\}},\dots, a_{\{N+1,N+2\}},b,v)  & \longmapsto & \frac{1}{2} \cdot \max \{\pm( \frac{1}{2}(\sum_{k=2 }^{N} \left(a_{\{1,k\}}-a_{\{k,N+1\}}\right)v_k(i)\\
& &+\left(a_{\{1,N+2\}}-a_{\{N+1,N+2\}}\right)(-v(i))\\
& &+(a_{\{1,N+1\}}) v(i))\\
& &-\frac{1}{2}(\sum_{k=2 }^{N} \left(a_{\{1,k\}}-a_{\{k,N+2\}}\right)\cdot v_k(i)\\
& & +\left(a_{\{1,N+1\}}- a_{\{N+1,N+2\}}\right)v(i)\\
& &+(a_{\{1,N+2\}})\cdot (-v(i))))\}
\end{eqnarray*}
which is rational and invariant under $G_N^{\lab}$.
\end{proposition}

\begin{proof}
The only thing to do is to show that this is indeed a function, i.e. it is well defined. For this we have to show, that $\phi_i$ is invariant under the addition of $c\cdot (s,0,0)$ (we identify $(s, 0, 0)$ with $s$) for $c\in \RR$ and the actions of $I$ and $M_p$. Let $x\in \tilde{X}_{n,\Delta,r}^{\lab}$ and $d=\phi_i(x)$.
 
 Let $c \in \RR$. Then, the value of $c\cdot s+x$ under $\phi_i$ is $d\pm \sum_{k=2}^N \left( -c\right) \cdot v_k(i)$. The second part ($\sum_{k=2}^N \left( -c\right) \cdot v_k(i)$) is $0$ due to the balancing condition, thus the value of $x$ and $c\cdot s+x$ is the same as before.

For $I$ we get the same, because
\begin{eqnarray*}
 & &\phi_i(I(a_{\{1,2\}},\dots, a_{\{N+1,N+2\}},b,v))\\ &=&\frac{1}{2} \cdot \max \{\pm( \frac{1}{2}(\sum_{k=2 }^{N} \left(a_{\{1,k\}}-a_{\{k,N+2\}}\right)v_k(i)\\
& &+\left(a_{\{1,N+1\}}-a_{\{N+1,N+2\}}\right)(-(-v(i)))+(a_{\{1,N+2\}})\cdot -v(i))\\
& & -\frac{1}{2}(\sum_{k=2 }^{N} \left(a_{\{1,k\}}-a_{\{k,N+1\}}\right)v_k(i)\\
& & +\left(a_{\{1,N+2\}}- a_{\{N+1,N+2\}}\right)(-v(i))+(a_{\{1,N+1\}})\cdot (-(-v(i)))))\}\\
&=& \frac{1}{2} \cdot \max \{\pm(-(\frac{1}{2}(\sum_{k=2 }^{N} \left(a_{\{1,k\}}-a_{\{k,N+1\}}\right)v_k(i)\\
& &+\left(a_{\{1,N+2\}}-a_{\{N+1,N+2\}}\right)(-v(i))+(a_{\{1,N+1\}})\cdot v(i))\\
& & -\frac{1}{2}(\sum_{k=2 }^{N} \left(a_{\{1,k\}}-a_{\{k,N+2\}}\right)v_k(i)\\
& & +\left(a_{\{1,N+1\}}- a_{\{N+1,N+2\}}\right)v(i)+(a_{\{1,N+2\}})\cdot (-v(i)))))\}\\
&=&\phi_i(a_{\{1,2\}},\dots, a_{\{N+1,N+2\}},b,v).
\end{eqnarray*}

 It remains to show the invariance with respect to $M_p$. Let us consider first the case $p \neq 1$.\\
We get:
 $$d\pm\frac{1}{4}\left(\left( (a_{\{1,N+1\}}+a_{\{p,N+2\}}+a_{\{N+1,N+2\}}-a_{\{1,N+2\}}-a_{\{p,N+1}\})\right.\right.$$
$$\left. +(a_{\{N+1,N+2\}}) \right) \cdot v_p(i) 
+\left(a_{\{1,N+2\}}-a_{\{N+1,N+2\}}\right)(v_p(i))+(a_{\{1,N+1\}})$$
$$\cdot (-v_p(i)) -\left( (a_{\{1,N+1\}}+a_{\{p,N+2\}}+a_{\{N+1,N+2\}}-a_{\{1,N+2\}}-a_{\{p,N+1\}}) \right.$$
$$\left.-(a_{\{N+1,N+2\}}) \right)\cdot v_p(i)+\left.\left(a_{\{1,N+1\}} - a_{\{N+1,N+2\}}\right)(v_p(i))-(a_{\{1,N+2\}})\cdot (v_p(i))\right)$$ $$=d.
 $$
In the case $p=1$, we have:
$$
d\pm\frac{1}{4}(\sum_{k=2 }^{N}(a_{k,N+1}+a_{1,N+2}+a_{\{N+1,N+2\}}-a_{\{k,N+2\}}-a_{\{1,N+1}\})\cdot v_k(i)+$$
$$ \left(a_{N+1,N+2}\right)\cdot (-v(i))+\left(a_{N+1,N+2}\right)\cdot (-v(i))$$
$$ -\sum_{k=2 }^{N}(a_{k,N+1}+a_{1,N+2}+a_{\{N+1,N+2\}}-a_{\{k,N+2\}}-a_{\{1,N+1}\})\cdot v_k(i)-$$
$$\left.\left(-a_{N+1,N+2}\right)\cdot (v(i))-\left(a_{N+1,N+2}\right)\cdot (-v(i))\right)=d.$$
Thus, $\phi_i$ is a rational function.
\end{proof}

\begin{remark}
We multiply the function by $\frac{1}{2}$, because locally the condition that the cycle closes leads to the function $\max$ $\{( \frac{1}{2}$ $\sum_{k=2 }^{N}$ $\left(a_{\{1,k\}}\right.$ $\left.-a_{\{k,N+1\}}\right)$ $v_k(i)$ $+\left(a_{\{1,N+2\}}\right.$ $\left.-a_{\{N+1,N+2\}}\right)$ $(-v(i))$ $+(a_{\{1,N+1\}})$ $\cdot v(i),0\}$. We changed the function slightly because of the symmetry we need for the orbit space structure.
\end{remark}

Now we can define the tropical orbit space we are interested in by constructing the tropical orbit space cut out by the rational functions $\phi_i$:
$$ \calM_{1,n,\trop}^{\text{lab}}(\mathbb{R}^r,\Delta)=  \phi_1\cdots \phi_r(\tilde{X}_{n,\Delta,r}^{\lab}/G_{N}^{\lab}), \text{ see Definition \ref{tropicalorbitspace}}.$$
The set of cones of $\calM_{1,n,\trop}^{\text{lab}}(\mathbb{R}^r,\Delta)$ is denoted by $X_{n,\Delta,r}^{\lab}$.
The rational functions assure that $A$ and $B$ are mapped to the same point.

\begin{example}
Let us consider the following map:
$$
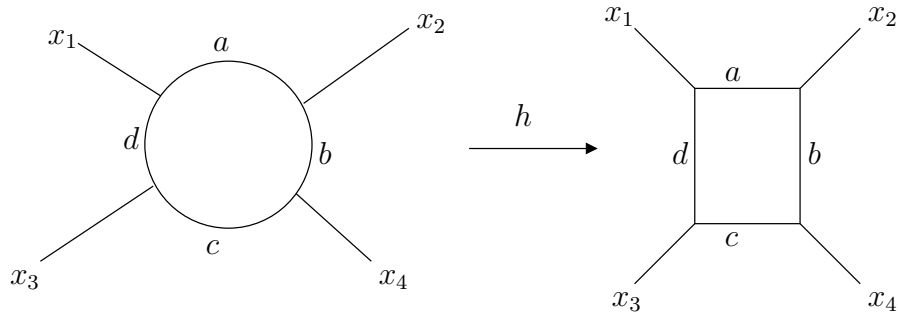
\begin{minipage} {\linewidth}
\begin{center}
\input{closedcycle.pstex_t}
\end{center}
\end{minipage}
$$
To ensure that $h$, defined by $h(x_3)=\binom{0}{0}$, $h(x_1)=d\cdot \binom{0}{1}$, $h(x_2)=d\cdot \binom{0}{1}+ a \cdot \binom{1}{0}$ and $h(x_4)=c\cdot \binom{1}{0}$, is the map of a tropical curve we need $a=c$ and $b=d$, which is the case for elements of $\calM_{1,n,\trop}^{\text{lab}}(\mathbb{R}^r,\Delta)$ due to the fact that the direction vectors multiplied by the lengths sum up to zero.
\end{example} 

The rational functions $\phi_i$ define weights on the resulting facets on the divisor. Since the stabilizers are finite the divisor is a tropical orbit space as well. Let us now consider the case $r=2$. The weights we get from the definition of the rational function are the following.
\begin{enumerate}
 \item  The image of the cycle is two-dimensional:\\
	the condition, that the cycle closes up in $\RR^2$ is given by two independent linear equations $a_1$ and $a_2$ on the lengths of the edges of the cycle (which is a subset of the bounded edges); thus, the weight is given by the index of the map:
	$$
	\left( \begin{array}{c} a_1 \\ a_2 \end{array} \right): \ZZ^{2+\# \Gamma^1_0} \mapsto \ZZ^2.
	$$
 \item  The image of the cycle is one-dimensional:\\
	because of the chosen rational function, we obtain that there has to be one four-valent vertex on the cycle. If not, the weight would be zero on the corresponding face. Let $m\cdot u$ and $n \cdot u$ with $u\in \ZZ^2, m,n \in \ZZ,$ and $gcd(n,m)=1$ be the direction vectors of the cycle (this is the same notation as in \cite{KM}). If we denote by $v\in \ZZ^2$ the direction of one other edge adjacent to the $4$-valent vertex, the weight is $|\det(u,v)|$. If $n=m=1$ and no point lies on the circle, the stabilizer of the corresponding face consists of two elements. Thus, the weight of the orbit space has to be divided by $2$ in this case. 
 \item	The image of the circle is 0-dimensional. 
	Due to the rational function we get the weight $\frac{1}{2}\cdot |\det(u,v)|$ (notation as in \cite{KM}) if there is a $5-$valent vertex adjacent to the cycle, $u,v$ are two of the three non-cycle directions outgoing from the vertex. If there is no $5-$valent vertex the weight would be zero by the definition of the rational function.
\end{enumerate}

\begin{proposition}\label{evaluation}
 For $ i=1,\dots,n $ the map
\begin{eqnarray*}
    \ev_i:
      X_{n,\Delta,r}^{\text{lab}}&\rightarrow&\RR^r\\
      (\Gamma,x_1,\dots x_N,h) & \longmapsto & h(x_i)
  \end{eqnarray*}
is invariant under the group operations.
\end{proposition}

\begin{proof}
The map $\ev_i$ is given by 
\begin{eqnarray}
\ev_i(x)&=&b+\frac{1}{2}\left(\sum_{k=2 }^{N} \left(a_{\{1,k\}}-a_{\{k,i\}}\right)v_k+\left(a_{\{1,N+1\}}-a_{\{N+1,i\}}\right)(v)\right.\nonumber\\ & &\left.+(a_{\{1,N+2\}}-a_{\{i,N+2\}})\cdot(-v)\vphantom{\sum_{k=2 }^{N}}\right).
\end{eqnarray}
Recall that $b=h(x_1)$. It is invariant under $s$, because the value added by $s$ to the differences $a_{\{1,N+1\}}$ $-a_{\{N+1,i\}}$ and $a_{\{1,N+2\}}-a_{\{i,N+2\}}$ is $0$.

The map $I$ changes only the order of the two last summands.

Thus, it remains to consider the map $M_p$. We have three cases: $p=1,p=i,1\neq p\neq i$. The sum we get differs from $(1)$ by the following expressions. 
Case $1\neq p\neq i$:
$$\frac{1}{2}\left(a_{\{1,N+1\}}+a_{\{p,N+2\}}+a_{\{N+1,N+2\}}-a_{\{1,N+2\}}-a_{\{p,N+1\}}-\right.$$
$$\left. (a_{\{i,N+1\}}+a_{\{p,N+2\}}+a_{\{N+1,N+2\}}-a_{\{i,N+2\}}-a_{\{p,N+1\}})\right)\cdot v_p$$
$$ +\frac{1}{2}\left(a_{\{1,N+1\}}-a_{\{N+1,i\}}\right)(-v_p)+\frac{1}{2}(a_{\{1,N+2\}}-a_{\{i,N+2\}})\cdot(v_p)=0.$$
Case $p=1$:
$$\sum_{k=2 }^{N} \frac{1}{2}\left(a_{\{k,N+1\}}+a_{\{1,N+2\}}+a_{\{N+1,N+2\}}-a_{\{k,N+2\}}-a_{\{1,N+1\}}\right)\cdot v_k+$$
$$\frac{1}{2}\left(-a_{\{N+1,N+2\}}\right)\cdot (v-v_1)+\frac{1}{2}\left(a_{\{1,N+1\}}-a_{\{N+1,i\}}\right)(-v_1)$$ $$+\frac{1}{2}\left(a_{\{N+1,N+2\}}\right)\cdot (-v+v_1)+\frac{1}{2}\left(a_{\{1,N+2\}}-a_{\{i,N+2\}}\right)\cdot(v_1)=0.$$
The last equation is true, because 
$$\sum_{k=2}^N (a_{1,N+2}+a_{N+1,N+2}-a_{1,N+1})v_k=0, v_1=0$$ 
and the rest of the sum 
\begin{multline}\left( \sum_{k=2 }^{N} \frac{1}{2}\left(a_{\{k,N+1\}}-a_{\{k,N+2\}}\right)\cdot v_k+ \frac{1}{2}\left(-a_{\{N+1,N+2\}}\right)\cdot (v)\right.\\ \left.
+\frac{1}{2}\left(a_{\{N+1,N+2\}}\right)\cdot (-v)\right) \nonumber
\end{multline} 
is equal to 
\begin{multline}
-\frac{1}{2}\left(\sum_{k=2 }^{N} \left(a_{\{1,k\}}-a_{\{k,N+1\}}\right)v_k
+\left(a_{\{1,N+2\}}\right.\left. -a_{\{N+1,N+2\}}\right)(-v)\right.\\ \left.+(a_{\{1,N+1\}})\cdot v\vphantom{\sum_{k=2 }^{N}}\right) +\left(\frac{1}{2}\left(\sum_{k=2 }^{N} \left(a_{\{1,k\}}-a_{\{k,N+2\}}\right)v_k\right.\right. \\
\left.\left. +\left(a_{\{1,N+1\}}- a_{\{N+1,N+2\}}\right)v+(a_{\{1,N+2\}})\cdot (-v)\vphantom{\sum_{k=2 }^{N}}\right)\vphantom{\frac{1}{2}}\right)\nonumber
\end{multline}
which is $0$ because of the rational function which we have used to constract $X_{n,\Delta,r}^{\lab}$ (see proposition \ref{rationalmap}).

Case $p=i$: 
$$ \frac{1}{2}\sum_{k=2 }^{N} -\left(a_{\{k,N+1\}}+a_{\{i,N+2\}}+a_{\{N+1,N+2\}}-a_{\{k,N+2\}}-a_{\{i,N+1\}}\right)\cdot v_k+$$ 
$$\frac{1}{2}\left(a_{\{N+1,N+2\}}\right)\cdot (v-v_i)+\frac{1}{2}\left(a_{\{1,N+1\}}-a_{\{N+1,i\}}\right)(-v_i)$$
$$+\frac{1}{2}\left(-a_{\{N+1,N+2\}}\right)\cdot (-v+v_i) +\frac{1}{2}\left(a_{\{1,N+2\}}-a_{\{i,N+2\}}\right)\cdot(v_i)=0.$$
(Same reason as above.)
\end{proof}

\begin{definition}[Evaluation map] \label{def-ev}
  For $ i=1,\dots,n $ the map
  \begin{eqnarray*}
    \ev_i:
      \calM_{1,n,\trop}^{\text{lab}}(\mathbb{R}^r,\Delta)&\rightarrow&\RR^r\\
      (\Gamma,x_1,\dots x_N,h) & \longmapsto & h(x_i)
  \end{eqnarray*}
  is called the {\em $i$-th evaluation map} (note that this is well-defined
  for the contracted ends since for them $ h(x_i) $ is a point in $ \RR^r $).
\end{definition}

\begin{proposition} \label{prop-evmorphism}
  With the tropical orbit space structure given above the
  evaluation maps $ \ev_i: \calM_{1,n,\trop}^{\lab}(\mathbb{R}^r,\Delta)
  \rightarrow \RR^r$ are morphisms of orbit spaces (in the sense of Definition 
\ref{def-morphism} and $\RR^r$ equipped with the trivial orbit space structure).
\end{proposition}

\begin{proof}
Since continuity is clear, we have to check conditions $a-d$ in Definition \ref{def-morphism}. Since we can move the curve arbitrarily in $\RR^r$, condition $a$ is clear. Condition $b$ is the same as the case of fans treated in \cite{GKM}. Condition $c$ is clear since each cone is mapped to the whole $\RR^r$ and the last condition follows from Proposition \ref{evaluation}.
\end{proof}

\begin{proposition}\label{morphism}
The map $f=\ev_1 \times \cdots \times ev_n \times j:\calM_{1,n,\trop}^{\lab}(\mathbb{R}^r,\Delta)\rightarrow \RR^{(rn+1)}$ is a morphism of orbit spaces.
\end{proposition}

\begin{proof}
Since all open ends in the cones of the moduli space are coming from the limit of the $j$-invariant to $0$, condition $c$ of Definition \ref{def-morphism} is fulfilled.
Thus, the statement follows from Proposition \ref{prop-evmorphism} and the fact that $j$ is the projection on the coordinate $\RR_{\{A,B\}}$. 

\end{proof}

\begin{theorem}
Let $d\geq 1$ and $n=3d-1$. Then the number of parameterized labeled $n$-marked tropical curves of genus $1$ and of degree $d$ with fixed j-invariant which pass through $n$ points in general position in $\RR^2$ is independent of the choice of the configuration of points.
\end{theorem}

\begin{proof}
For $n=3d-1$ points $\calM_{1,n,\trop}^{\lab}(\mathbb{R}^2,d)$ has the same dimension as $\RR^{(rn)}\times \RR_{>0}$. Since all open ends are mapped to $j$-invariant equal $0$, surjectivity follows by the balancing condition. Thus, Proposition \ref{morphism} and Corollary \ref{cor-image} imply the theorem.
\end{proof}

\end{document}